\documentclass{amsart} 
\usepackage{amssymb}
\usepackage{amsmath}
\usepackage{amsfonts}

\begin{document}
\newtheorem{theo}{Theorem}[section]
\newtheorem{atheo}{Theorem*}
\newtheorem{prop}[theo]{Proposition}
\newtheorem{aprop}[atheo]{Proposition*}
\newtheorem{lemma}[theo]{Lemma}
\newtheorem{alemma}[atheo]{Lemma*}
\newtheorem{exam}[theo]{Example}
\newtheorem{coro}[theo]{Corollary}
\theoremstyle{definition}
\newtheorem{defi}[theo]{Definition}
\newtheorem{rem}[theo]{Remark}


\newcommand{\Bb}{{\bf B}}
\newcommand{\Cb}{{\mathbb C}}
\newcommand{\Nb}{{\mathbb N}}
\newcommand{\Qb}{{\mathbb Q}}
\newcommand{\Rb}{{\mathbb R}}
\newcommand{\Zb}{{\mathbb Z}}
\newcommand{\Ac}{{\mathcal A}}
\newcommand{\Bc}{{\mathcal B}}
\newcommand{\Cc}{{\mathcal C}}
\newcommand{\Dc}{{\mathcal D}}
\newcommand{\Fc}{{\mathcal F}}
\newcommand{\Ic}{{\mathcal I}}
\newcommand{\Jc}{{\mathcal J}}
\newcommand{\Kc}{{\mathcal K}}
\newcommand{\Lc}{{\mathcal L}}
\newcommand{\Oc}{{\mathcal O}}
\newcommand{\Pc}{{\mathcal P}}
\newcommand{\Sc}{{\mathcal S}}
\newcommand{\Tc}{{\mathcal T}}
\newcommand{\Uc}{{\mathcal U}}
\newcommand{\Vc}{{\mathcal V}}

\author{Charles Akemann and Nik Weaver}

\title [A Lyapunov type theorem]
       {A Lyapunov type theorem from Kadison-Singer}

\address {Department of Mathematics\\
UC Santa Barbara\\
Santa Barbara, CA 93106\\
Department of Mathematics\\
Washington University\\
Saint Louis, MO 63130}

\email {akemann@math.ucsb.edu, nweaver@math.wustl.edu}

\date{\em August 12, 2013}

\thanks{Second author supported by NSF grant DMS-1067726.}


\begin{abstract}
Marcus, Spielman, and Srivastava \cite{MSS} recently solved the Kadison-Singer
problem by showing that if $u_1, \ldots, u_m$ are column vectors in $\Cb^d$
such that $\sum u_iu_i^* = I$, then a set of indices $S \subseteq \{1, \ldots,
m\}$ can be chosen so that $\sum_{i \in S} u_iu_i^*$ is approximately
$\frac{1}{2}I$, with the approximation good in operator norm to order
$\epsilon^{1/2}$ where $\epsilon = \max \|u_i\|^2$. We extend their result
to show that every linear combination of the matrices $u_iu_i^*$ with
coefficients in $[0,1]$ can be approximated in operator norm to
order $\epsilon^{1/8}$ by a matrix of the form $\sum_{i \in S} u_iu_i^*$.
\end{abstract}

\maketitle

Marcus, Spielman, and Srivastava recently proved the following theorem
(\cite{MSS}, Corollary 1.3):
\begin{theo}\label{mssthm}
Let $u_1, \ldots, u_m$ be column vectors in $\Cb^d$ such that
$\sum_i u_iu_i^* = I$ and $\|u_i\|^2 \leq \epsilon$ for all $i$.
Then there exists a partition of $\{1, \ldots, m\}$ into sets
$S_1$ and $S_2$ so that for $j \in \{1,2\}$,
$$\left\|\sum_{i \in S_j} u_iu_i^*\right\| \leq
\frac{1}{2} + \sqrt{2\epsilon} + \epsilon.$$
\end{theo}
\noindent Here $\|u_i\|$ is euclidean norm in $\Cb^d$ and
$\|\sum_{i \in S_j} u_iu_i^*\|$ is operator norm in $M_d(\Cb)$.
Thanks to previous reductions due to Akemann-Anderson \cite{AA} and
Weaver \cite{W}, this stunning result implies a positive solution to
the celebrated Kadison-Singer problem (\cite{KS}; see \cite{CC} for
background). A more elementary approach to the Akemann-Anderson
portion of the reduction can be found in \cite{B}.

If $u$ is a vector in $\Cb^d$ then $uu^*$ is a positive rank one matrix.
It is the orthogonal projection onto the one-dimensional subspace of $\Cb^d$
spanned by $u$, scaled by the factor $\|u\|^2$. So if $\epsilon$ is small,
then the hypothesis $\sum u_iu_i^* = I$ means that we have a large number
of small multiples of rank one projections which sum up to the identity
matrix. The conclusion that $\|\sum_{i \in S_j} u_iu_i^*\| \leq
\frac{1}{2} + O(\sqrt{\epsilon})$ for $j \in \{1,2\}$ means that the positive matrix
$A = \sum_{i \in S_1} u_iu_i^*$ satisfies
$\|A\| \leq \frac{1}{2} + O(\sqrt{\epsilon})$ and
$\|I - A\| \leq \frac{1}{2} + O(\sqrt{\epsilon})$, i.e.,
$$\left(\frac{1}{2} - O(\sqrt{\epsilon})\right)I \leq A \leq
\left(\frac{1}{2} + O(\sqrt{\epsilon})\right)I.$$
We are using the order relation on matrices which puts $A \leq B$ if
$B - A$ is positive semidefinite.
Another way to say this is that $\|A - \frac{1}{2}I\| = O(\sqrt{\epsilon})$.
Thus the Marcus-Spielman-Srivastava theorem tells us that
we can partition any set of small multiples of rank one
projections which sum up to the identity matrix, into two subsets each of
which sums up to approximately half the identity matrix.

(One could ask whether this result is still true if we replace ``small
multiples of rank one projections'' with ``positive matrices of small
norm''. It is not; the more general assertion fails even for diagonal
matrices. A better conjecture is that it remains true for positive matrices
with small trace. Diagonal matrices will not refute this assertion, by
the continuous Beck-Fiala theorem \cite{AA2}.)

The main result of this note is a modest extension to Theorem
\ref{mssthm}. A trivial modification of the argument of
\cite{MSS} shows that for any $t \in [0,1]$ we can find a set of
indices $S \subseteq \{1, \ldots, m\}$ such that
$$\left\|\sum_{i \in S} u_iu_i^* - tI\right\| = O(\sqrt{\epsilon});$$
this led us to ask which matrices $0 \leq B \leq I$ can be approximated
by a sum of the form $\sum_{i \in S} u_iu_i^*$ for suitable $S$.

Here is a simple example which shows that not every positive matrix
of norm at most 1 can be
approximated. Take $d = 2$, fix $M \in \Nb$, and define $u_1 =
\cdots = u_M = \frac{1}{\sqrt{M}}e_1$ and $u_{M+1} = \cdots =
u_{2M} = \frac{1}{\sqrt{M}}e_2$, where $e_1 =
\begin{pmatrix}
1\cr
0\end{pmatrix}$ and $e_2 =
\begin{pmatrix}
0\cr
1\end{pmatrix}$ are the standard basis vectors in $\Cb^2$. Then
$$u_iu_i^* =
\begin{cases}
\begin{pmatrix}
\frac{1}{M}&0\cr
0&0
\end{pmatrix}&1 \leq i \leq M\cr
\begin{pmatrix}
0&0\cr
0&\frac{1}{M}
\end{pmatrix}&M+1 \leq i \leq 2M
\end{cases}$$
so that $\sum u_iu_i^* = I$, as desired. Also $\|u_i\|^2 = \frac{1}{M}$
for all $i$. However, for any set of indices $S$, the matrix
$\sum_{i \in S}u_iu_i^*$ is diagonal, so the only matrices $0 \leq B
\leq I$ that can be approximated by a sum of this form are those which
are nearly diagonal.

We will prove the following result. Given any vectors $u_1, \ldots, u_m$
as in Theorem \ref{mssthm} and any choice of coefficients
$t_i \in [0,1]$ for $1 \leq i \leq m$, there is a set of indices $S$ such
that the matrix $\sum_{i \in S} u_iu_i^*$ approximates the matrix
$\sum t_i u_iu_i^*$ in operator norm to order $\epsilon^{1/8}$.

The set of matrices we claim can be approximated is simply the
convex hull of the set of matrices of the form $\sum_{i \in S} u_iu_i^*$ with
$S \subseteq \{1,\ldots, m\}$. So in effect we are saying that the set
of matrices of this form is already, in a sense, approximately convex.
This result can be viewed more abstractly as a ``Lyapunov type'' theorem,
as we will now explain.

Regard a choice of coefficients $t_1, \ldots, t_m \in [0,1]$ as
a single point $t \in [0,1]^m$. Then the map
$$\Phi: t \mapsto \sum t_iu_iu_i^*$$
takes the unit cube $[0,1]^m$ to a compact convex subset of the space
$M_d(\Cb)$ of $d\times d$ complex matrices. Meanwhile, any sum of the
form $\sum_{i \in S} u_iu_i^*$ is the image under $\Phi$ of one of
the vertices of the cube. So our result says that the image of $[0,1]^m$
is approximated by the image of its vertex set $\{0,1\}^m$.

The classical Lyapunov theorem \cite{lyap} states that the range of a
nonatomic vector-valued measure is compact and convex. In Lindenstrauss's
formulation \cite{lind}, we have a weak*-continuous linear map $\Psi:
L^\infty[0,1] \to \Rb^n$. The range of the vector-valued measure is the
set $\{\Psi(\chi_A): A \subseteq [0,1]$ measurable $\}$, the image of the
characteristic functions in $L^\infty[0,1]$ under this map. It is not
obvious that this set is compact or convex, but it is easy to see that
the image under $\Psi$ of the positive part of the unit ball,
$\{f \in L^\infty[0,1]: 0 \leq f \leq 1\}$, is compact and convex, and
the characteristic functions in $L^\infty[0,1]$ are the extreme
points of this set. So Lindenstrauss's version of Lyapunov's theorem
states that in this setting the image of the positive part of the unit
ball of $L^\infty[0,1]$
equals the image of its set of extreme points. In \cite{AA} the term
{\it Lyapunov theorem} was used to mean any result which states that for
some convex set $Q$ and some affine map of $Q$ into a linear space, the
image of $Q$ equals the image of its set of extreme points.

The extreme points of the cube $[0,1]^m$ are its vertices. So we can
consider our result, which states that everything in the image of
$[0,1]^m$ under the map $\Phi$ is approximated by the image of a
vertex, as an approximate Lyapunov type theorem. 

Our result can also be viewed from the perspective of control theory
as relating to the notion of a ``bang-bang control'' \cite{HLS}.

\section{The main result}

Our main result is achieved by a series of successive approximations.
We start with a straightforward generalization of (\cite{MSS}, Corollary 1.3).
Although we only need the special case where $r = 2$, we state it for general
$r$ for the sake of completeness.

\begin{lemma}\label{lem1}
Let $u_1, \ldots, u_m$ be column vectors in $\Cb^d$ such that
$\sum_i u_iu_i^* = I$ and $\|u_i\|^2 \leq \epsilon$ for all $i$. Also let
$r \in \Nb$ and suppose $t_1, \ldots, t_r > 0$ satisfy $\sum_j t_j = 1$.
Then there exists a partition of $\{1, \ldots, m\}$ into $r$ sets
$S_1, \ldots, S_r$ so that for $1 \leq j \leq r$,
$$\left\|\sum_{i \in S_j} u_iu_i^*\right\| \leq
t_j(1 + \sqrt{r\epsilon})^2.$$
\end{lemma}

\begin{proof}
We modify the proof of (\cite{MSS}, Corollary 1.3). Let $v_1, \ldots, v_m$
be independent random vectors in $\Cb^{rd} \cong \Cb^d \oplus \cdots \oplus
\Cb^d$ ($r$ summands) such that for each $i$, with probability $t_j$ the
vector $v_i$ is $t_j^{-1/2}u_i$ in the $j$th summand and zero in all other
summands. Then 
$$E \|v_i\|^2 = \sum_{j=1}^r t_j\cdot t_j^{-1}\|u_i\|^2 = r\|u_i\|^2$$
and $Ev_iv_i^*$ is the block diagonal matrix $u_iu_i^* \oplus \cdots \oplus
u_iu_i^*$, so that
$$\sum_{i=1}^m Ev_iv_i^* = I_{rd}.$$
So (\cite{MSS}, Theorem 1.2) applies with $rd$ in place of $d$ and
$r\epsilon$ in place of $\epsilon$. We conclude that there is a partition
of $\{1, \ldots, m\}$ into $r$ sets $S_1, \ldots, S_r$ such that
$\|\sum_{i \in S_j} t_j^{-1} u_iu_i^*\| \leq (1 + \sqrt{r\epsilon})^2$
for $1 \leq j \leq r$.
\end{proof}

For each $j$ Lemma \ref{lem1} yields both the inequality
$$\sum_{i \in S_j} u_iu_i^* \leq t_j(1 + \sqrt{r\epsilon})^2I
= (t_j + t_j(2\sqrt{r\epsilon} + r\epsilon))I$$
and,
summing over $\bigcup_{j' \neq j} S_{j'}$, the inequality
$$I - \sum_{i \in S_j} u_iu_i^* = \sum_{i \not\in S_j} u_iu_i^*
\leq (1 - t_j)(1 + \sqrt{r\epsilon})^2I.$$
Thus
$$\sum_{i \in S_j} u_iu_i^* \geq (1 - (1-t_j)(1 + \sqrt{r\epsilon})^2)I
= (t_j - (1 - t_j)(2\sqrt{r\epsilon} + r\epsilon))I,$$
and by averaging the upper and lower bounds we get the following
estimate.

\begin{coro}\label{lemcor}
Under the same hypotheses as in Lemma \ref{lem1}, we have
$$\left\|\sum_{i \in S_j} u_iu_i^* - t_j'I\right\| \leq
\sqrt{r\epsilon} + \frac{1}{2}r\epsilon$$
for each $j$, where
$t_j' = t_j + (t_j - \frac{1}{2})(2\sqrt{r\epsilon} + r\epsilon)$.
\end{coro}

So the sums
$\sum_{i \in S_j} u_iu_i^*$ approximately equal the matrices $t_jI$.
In particular, taking $r = 2$ and letting $t_1 = t$ and $t_2 = 1 - t$, there
exists a set of indices $S$ such that
$$\left\|\sum_{i \in S} u_iu_i^* - tI\right\| = O(\sqrt{\epsilon}).$$

Next we generalize to the case where $\sum u_iu_i^*$ is less than
the identity matrix. The idea is to ignore regions where $\sum u_iu_i^*$ is
less than $\sqrt{\epsilon}$, and in regions where it is greater than
$\sqrt{\epsilon}$ to scale it up to 1. This entails some loss of accuracy.
From here on we restrict to the case $r = 2$.

\begin{lemma}\label{lem2}
Let $u_1, \ldots, u_m$ be column vectors in $\Cb^d$ such that
$B = \sum_i u_iu_i^* \leq I$ and $\|u_i\|^2 \leq \epsilon$ for all $i$.
Suppose $0 \leq t \leq 1$. Then there is a set of indices
$S \subseteq \{1, \ldots, m\}$ such that
$$\left\|\sum_{i \in S}u_iu_i^* - tB\right\| = O(\epsilon^{1/4}).$$
\end{lemma}

\begin{proof}
Let $E \subseteq \Cb^d$ be the spectral subspace of $B$ for the interval
$[\sqrt{\epsilon}, 1]$ and let $P_E$ be the orthogonal projection onto $E$,
so that $\sqrt{\epsilon}P_E \leq BP_E \leq P_E$. Define $C = B^{-1/2}P_E$,
so that $P_E \leq C \leq \epsilon^{-1/4}P_E$, and set $v_i = Cu_i$ for all
$i$. Then $v_i \in E$ and $\|v_i\|^2 \leq \sqrt{\epsilon}$ for all $i$, and
$$\sum v_iv_i^* = C\left(\sum u_iu_i^*\right)C = CBC = P_E.$$
So by Corollary \ref{lemcor} there exists a set of indices $S$ such that
$$\left\|\sum_{i \in S} v_iv_i^* - tP_E\right\| = O(\epsilon^{1/4}).$$
Now let $D = \sum_{i \in S} u_iu_i^*$ and observe that
$$\|P_E(D - tB)P_E\| = \left\|B^{1/2}\left(\sum_{i \in S}v_iv_i^* -
tP_E\right)B^{1/2}\right\| = O(\epsilon^{1/4}).$$
Also $0 \leq D \leq B$, so that $(I - P_E)D(I - P_E) \leq
(I - P_E)B(I - P_E) \leq \sqrt{\epsilon}I$. Finally, if $v \in E$ and
$w \in E^\perp$ then
$$|\langle Dv,w\rangle| = |\langle v,Dw\rangle| =
|\langle D^{1/2}v,D^{1/2}w\rangle| \leq \|v\|\cdot \epsilon^{1/4}\|w\|,$$
so that $\|P_E D(I - P_E)\|, \|(I - P_E)DP_E\| \leq \epsilon^{1/4}$.
Since $P_EtB(I - P_E) = (I - P_E)tBP_E = 0$ and $\|(I - P_E)tB(I - P_E)\|
\leq \sqrt{\epsilon}$, we can conclude that $\|D - tB\| = O(\epsilon^{1/4})$.
\end{proof}

We now come to our main result.

\begin{theo}\label{main}
Let $u_1, \ldots, u_m$ be column vectors in $\Cb^d$ such that
$\sum_i u_iu_i^* \leq I$ and $\|u_i\|^2 \leq \epsilon$ for all $i$.
Suppose $0 \leq t_i \leq 1$ for $1 \leq i \leq m$. Then there is a set
of indices $S \subseteq \{1, \ldots, m\}$ such that
$$\left\|\sum_{i \in S}u_iu_i^* - \sum_{i=1}^m t_iu_iu_i^*\right\|
= O(\epsilon^{1/8}).$$
\end{theo}

\begin{proof}
Let $n$ be the nearest integer to $\epsilon^{-1/8}$ and partition
$\{1, \ldots, m\}$ into $n$ sets $S_1, \ldots, S_n$ such that
$|t_i - \frac{k}{n}| \leq \frac{1}{n}$ for all $i \in S_k$. For each
$1 \leq k \leq n$ and each $i \in S_k$, let $t_i' = \frac{k}{n}$.

For each $k$, by Lemma \ref{lem2} we can find a subset $S_k' \subseteq S_k$
such that $\|\sum_{i \in S_k'} u_iu_i^* - \frac{k}{n}\sum_{i\in S_k} u_iu_i^*\|
= O(\epsilon^{1/4})$. Let $S = \bigcup_{k = 1}^n S_k'$. Then
$$\left\|\sum_{i \in S} u_iu_i^* - \sum_{i=1}^m t_i'u_iu_i^*\right\|
\leq \sum_{k=1}^n\left\|\sum_{i \in S_k'} u_iu_i^*
- \frac{k}{n}\sum_{i \in S_k} u_iu_i^*\right\|
= n\cdot O(\epsilon^{1/4}) = O(\epsilon^{1/8})$$
so
$$\left\|\sum_{i \in S} u_iu_i^* - \sum_{i=1}^m t_iu_iu_i^*\right\|
\leq \left\|\sum_{i \in S} u_iu_i^* - \sum_{i=1}^m t_i'u_iu_i^*\right\|
+ \left\|\sum_{i=1}^m (t_i'-t_i)u_iu_i^*\right\| = O(\epsilon^{1/8}),$$
as desired.
\end{proof}

Thus, every linear combination of the matrices $u_iu_i^*$ with coefficients
in $[0,1]$ can be approximated to order $\epsilon^{1/8}$ by a sum of the form
$\sum_{i \in S} u_iu_i^*$.

\section{Infinite dimensions}

The power of the Marcus-Spielman-Srivastava result lies in the fact that
the estimate depends only on the size of the vectors $u_i$ (via the
parameter $\epsilon$), and not on the number of vectors or the dimension
of the space. Consequently, it should not be surprising that their result
generalizes to infinite-dimensional Hilbert spaces. Here we shift notation
and write $u\otimes u$ for the rank one positive operator
$v \mapsto \langle v,u\rangle u$, whenever $u$ is a vector in a complex
Hilbert space. Also, infinite sums of the form
$\sum u_i\otimes u_i$ should be understood in the sense of either strong
or weak operator convergence (the two are equivalent for bounded increasing
sequences). The infinite-dimensional version of Theorem \ref{mssthm} can
then be stated as follows:

\begin{theo}\label{inf}
Let $(u_i)$ be a sequence of vectors in a Hilbert space $H$ such that
$\sum_i u_i\otimes u_i = I$ and $\|u_i\|^2 \leq \epsilon$ for all $i$.
Then there is a partition of $\Nb$ into sets $S_1$ and $S_2$ so that
for $j \in \{1,2\}$,
$$\left\|\sum_{i \in S_j} u_i\otimes u_i\right\| \leq \frac{1}{2} +
\sqrt{2\epsilon} + \epsilon.$$
\end{theo}

The proof of Theorem \ref{inf} is fairly straightforward. We first consider
the case of an infinite sequence of vectors in a finite-dimensional space
$\Cb^d$. This case can be reduced to Theorem \ref{mssthm} by, for each
$k \in \Nb$, choosing an index $n_k$ such that
$B_k = \sum_{i=1}^{n_k} u_i \otimes u_i \geq (1 - \frac{1}{k})I$.
Then replacing each $u_i$, $1 \leq i \leq n_k$, with the vector
$B_k^{-1/2}u_i$ as in the proof of Lemma \ref{lem2} gives us the hypothesis
of Theorem \ref{mssthm}. This yields, for each $k$, a partition of the
index set $\{1, \ldots, n_k\}$. Finally, as $k \to \infty$ we can take a
cluster point of this sequence of partitions in the compact space
$\{0,1\}^\Nb$ (the space of all infinite 0-1 sequences).

We can then pass to the case of infinite-dimensional $H$ by projecting
the sequence $(u_i)$ into a finite-dimensional subspace, invoking the
finite-dimensional result just discussed, and applying another compactness
argument as that subspace increases to $H$. Extending the result to
uncountable families of vectors in nonseparable Hilbert spaces presents
no difficulties.

All of the results in Section 1 generalize in a similar way to infinite
dimensions, with the same estimates in every case. One might also
predict that in infinite dimensions factoring out compact operators
could allow us to convert the approximate results of Section 1 into
exact ones. This is indeed the case.

\begin{theo}\label{compact}
Let $(u_i)$ be a sequence of vectors in a complex Hilbert space such
that $\|u_i\| \to 0$ and $\sum u_i\otimes u_i \leq I$. Suppose
$0 \leq t_i\leq 1$ for each $i$. Then there is a set of indices
$S \subseteq \Nb$ such that the operator
$$\sum_{i \in S} u_i\otimes u_i - \sum_{i=1}^\infty t_iu_i\otimes u_i$$
is compact.
\end{theo}

\begin{proof}
Partition $\Nb$ into a sequence of finite sets $S_0, S_1, S_2, \ldots$ with
the property that $\|u_i\|^2 \leq 2^{-n}$ for all $i \in S_n$. Then for each
$n$ apply Theorem \ref{main} to find
a subset $S_n' \subseteq S_n$ such that $\|\sum_{i \in S_n'}u_i\otimes u_i
- \sum_{i \in S_n} t_iu_i\otimes u_i\| = O(2^{-n/8})$.

Let $S = \bigcup_{n=0}^\infty S_n'$. Given $N \in \Nb$, define
$\delta_N = \sum_{n=N+1}^\infty 2^{-n/8}$; then
$$\left\|\sum_{n=N+1}^\infty\left(\sum_{i \in S_n'}u_i\otimes u_i
- \sum_{i \in S_n} t_iu_i\otimes u_i\right)\right\| = O(\delta_N).$$
Since $\sum 2^{-n/8} < \infty$, this goes to zero as $N \to \infty$.
Thus we have shown that by subtracting the finite rank operator
$$\sum_{n=0}^N\left(\sum_{i \in S_n'}u_i\otimes u_i
- \sum_{i \in S_n} t_iu_i\otimes u_i\right),$$
the operator
$\sum_{i \in S} u_i\otimes u_i - \sum_{i=1}^\infty t_iu_i\otimes u_i$
can be made to have arbitrarily small norm. Therefore it is compact.
\end{proof}

This last result can be more elegantly expressed in terms of projections
in the Calkin algebra. We explain this alternative formulation in the
final section of the paper.

\section{Projection formulation}

In \cite{AA} the sorts of problems we have been discussing were
formulated in terms of projection matrices. Let $E$ be a linear subspace
of $\Cb^m$ and let $P: \Cb^m \to E$ be the orthogonal projection of $\Cb^m$
onto $E$. Then the vectors $u_i = Pe_i$, where $(e_i)$ is the standard basis
of $\Cb^m$, satisfy $\sum u_iu_i^* = P\sum e_ie_i^* P = P$. Moreover, the
diagonal entries of $P$ are the values $\langle Pe_i,e_i\rangle = \|Pe_i\|^2
= \|u_i\|^2$. So Theorem \ref{mssthm} can be interpreted in these terms.

Let $Q_i = e_ie_i^*$ be the matrix with a 1 in the $(i,i)$ entry and 0's
elsewhere. Then $u_iu_i^* = PQ_iP$. The sum $\sum_{i \in S} u_iu_i^*$
over a selected set of indices is therefore a matrix of the form $PQP$
where $Q$ is a {\it diagonal projection}, a diagonal matrix whose
entries are all either $0$ or $1$. Thus, Theorem \ref{mssthm} can be
expressed in the following form:

\begin{theo}\label{amssthm}
Let $P \in M_m(\Cb)$ be an $m \times m$ complex projection matrix whose
diagonal entries are each at most $\epsilon$. Then there exists a diagonal
projection $Q \in M_m(\Cb)$ such that
$$\|PQP\|, \|P(I-Q)P\| \leq \frac{1}{2} + \sqrt{2\epsilon} + \epsilon.$$
\end{theo}

As in the introduction, we can express this pair of inequalities by the
single statement that $\|PQP - \frac{1}{2}P\| \leq
\frac{1}{2} + \sqrt{2\epsilon} + \epsilon$.

The projection formulation is actually equivalent to the vector
formulation. This is because the condition $\sum u_iu_i^* = I$ says that
the vectors $u_i$ constitute a {\it tight frame}, and according to a
standard construction, every tight frame in $\Cb^d$ arises by orthogonally
projecting into $\Cb^d$ an orthonormal basis of some containing space.
See (\cite{W}, p.\ 230). Thus, we can pass back and forth between results
about frames in $\Cb^d$ and results about projection matrices in $\Cb^m$.
We will now summarize the results of Sections 1 and 2 in the language
of projection matrices.

The projection formulation of Lemma \ref{lem1} is a quantitative
version of (\cite{AA}, Theorem 7.12).

\begin{lemma}\label{alem1}
Let $P \in M_m(\Cb)$ be an $m \times m$ complex projection matrix whose
diagonal entries are each at most $\epsilon$. Also let $r \in \Nb$ and
suppose $t_1, \ldots, t_r > 0$ satisfy $\sum_j t_j = 1$.
Then there exist $r$ diagonal projections $Q_1, \ldots, Q_r$ in $M_m(\Cb)$
such that $\sum_{j=1}^r Q_j = I$ and
$$\|PQ_jP\| \leq t_j(1 + \sqrt{r\epsilon})^2$$
for $1 \leq j \leq r$.
\end{lemma}

As with Lemma \ref{lem1}, we can infer that
$\|PQ_jP  - t_j'P\| \leq \sqrt{r\epsilon} + \frac{1}{2}r\epsilon$
for $1 \leq j \leq r$.

In the projection version of Lemma \ref{lem2} the condition
$\sum u_iu_i^* \leq I$ corresponds to working under a diagonal projection
in the containing space:

\begin{lemma}\label{alem2}
Let $P \in M_m(\Cb)$ be an $m \times m$ complex projection matrix, let
$Q \in M_m(\Cb)$ be a diagonal projection, and let $0 \leq t \leq 1$. Assume
that the diagonal entries of $QPQ$ are each at most $\epsilon$, i.e.,
$\langle Pe_i,e_i\rangle \leq \epsilon$ for each basis vector $e_i$
such that $Qe_i = e_i$. Then there is a diagonal
projection $Q' \leq Q$ such that
$$\|PQ'P - tPQP\| = O(\epsilon^{1/4}).$$
\end{lemma}

Our main result takes the following form.

\begin{theo}\label{amain}
Let $P \in M_m(\Cb)$ be an $m \times m$ complex projection matrix,
let $Q \in M_m(\Cb)$ be a diagonal projection, and let $B \in M_m(\Cb)$ be
a diagonal matrix satisfying $0 \leq B \leq Q$.
Assume that the diagonal entries of $QPQ$ are each at most $\epsilon$.
Then there is a diagonal projection $Q' \leq Q$ such that
$$\|PQ'P - PBP\| = O(\epsilon^{1/8}).$$
\end{theo}

(In order to deduce Theorem \ref{main} from Theorem \ref{amain}, we need
to know that any finite set of vectors satisfying $B = \sum_i u_iu_i^* \leq I$
is contained in a tight frame, so that it arises from projecting into
$\Cb^d$ some subset of the standard basis of a containing space.
To see this, write the positive matrix
$I - B$ as a sum of positive rank one matrices and use the fact that any
positive rank one matrix has the form $uu^*$ for some vector $u$.)

The projection version of Theorem \ref{inf} is unsurprising.

\begin{theo}\label{ainf}
Let $P \in \Bc(l^2)$ be a projection whose diagonal entries are each
at most $\epsilon$. Then there exists a diagonal projection $Q$ such
that
$$\|PQP\|, \|P(I-Q)P\| \leq \frac{1}{2} + \sqrt{2\epsilon} + \epsilon.$$
\end{theo}

A similar finite- to infinite-dimensional inference was proven in
(\cite{AA}, Proposition 7.6), and our proof of Theorem \ref{inf} is
essentially a translation of the argument given there. Again, generalizing
Theorem \ref{ainf} to the nonseparable setting is straightforward.

Finally, as we mentioned at the end of Section 2, Theorem \ref{compact}
has an attractive formulation in terms of projections. Let
$\pi: \Bc(l^2) \to \Cc(l^2) = \Bc(l^2)/\Kc(l^2)$ be the natural projection
of $\Bc(l^2)$ onto the Calkin algebra, let $\Phi: \Bc(l^2) \to l^\infty$
be the conditional expectation of $\Bc(l^2)$ onto its diagonal subalgebra,
and let $\tilde{\Phi}: \Cc(l^2) \to l^\infty/c_0$ be the corresponding
conditional expectation of $\Cc(l^2)$ onto its diagonal subalgebra.

\begin{theo}\label{calkin}
Let $p \in \Cc(l^2)$ be a projection such that $\tilde{\Phi}(p) = 0$, let
$q \in \Cc(l^2)$ be a diagonal projection, and let $b \in \Cc(l^2)$ be a
diagonal operator satisfying $0 \leq b \leq q$. Then there is a diagonal
projection $q' \leq q$ such that $pq'p = pbp$.
\end{theo}

Theorem \ref{calkin} is a theorem of Lyapunov type. Letting $\Dc$ be the
diagonal subalgebra of $\Cc(l^2)$, we have that the set of diagonal
operators satisfying $0 \leq b \leq q$ is the positive part of the unit
ball of $q\Dc$. Theorem \ref{calkin} says that the image of any element
of this set under the map $b \mapsto pbp$ equals the image of an
extreme point.


\end{document}